\documentclass[a4paper,UKenglish]{lipicsown}

\usepackage{tikz}
\usepackage{url}
\usepackage{color}
\usepackage{hyperref}
\hypersetup{colorlinks   = true,
     citecolor    = black,
     urlcolor     = violet,
     linkcolor    = blue}

\usepackage{tikz}
\usetikzlibrary{arrows,automata}

\usepackage{amsfonts}
\usepackage{amssymb}
\usepackage{amsmath}
\usepackage{mathtools}
\usepackage{float}
\usepackage{tikz}

\newtheorem{claim}[theorem]{Claim}

\newtheorem{algo}[theorem]{Algorithm}
\newtheorem{rem}[theorem]{Remark}

\newtheorem{defn}[theorem]{Definition}

\title{Chains and Antichains inside Many-One Degrees and Variants}


\author[2]{Linus Richter$^1$, Frank Stephan$^{2,3}$ and Xiaoyan Zhang}
\affil[1]{Flinders University, College of Science and Engineering, Level 3,
Tonsley Building 1, 1284 South Road, Tonsley SA 5042, Australia;
\texttt{linus.richter@flinders.edu.au}.
Linus Richter did most of his work on this paper while he worked for
the Department of Mathematics, National University of Singapore, during
his previous employment.}
\affil[2]{Department of Mathematics,
National University of Singapore,
10 Lower Kent Ridge Road, Block S17, Singapore 119076,
Republic of Singapore,
\texttt{fstephan@nus.edu.sg} and \texttt{zhangxy@u.nus.edu}.}
\affil[3]{School of Computing, National University of
Singapore, 13 Computing Drive, Block COM1, Singapore 117417,
Republic of Singapore.}

\authorrunning{L.~Richter, F.~Stephan and Xy.~Zhang}

\titlerunning{Chains and Antichains inside Many-One Degrees and Variants}

\date{\today}

\begin{document}
\maketitle

\begin{abstract}
The relations between many-one degrees and one-one degrees have been
studied since the beginning of recursion theory; early results from
the 1960s include that many-one degrees always have a largest
one-one degree and either that one-one degree is the only one-one degree
inside the many-one degree or every countable linear order is noneffectively
embeddable into the structure of one-one degrees inside the given
many-one degree. Furthermore, the greatest recursive many-one degree
is a special case, as it allows to embed ascending infinite chains but
not descending infinite chains, all other many-one degrees fall into the
two cases mentioned above. It remained open whether infinite antichains
can always be embedded when the many-one degree is nonrecursive and
nonirreducible; Odifreddi stated in his survey from the year 1981 and
in his book Classical Recursion Theory in the
year 1989 this question explicitly as an open problem. D\"egtev
had already in 1976 constructed antichains of one-one degrees inside all
nonrecursive and nonirreducible recursively enumerable many-one degrees
and Batyrshin generalised the result to all nonrecursive and nonirreducible
limit-recursive many-one degrees. Recently, Cintioli \cite{Ci2026}
showed that there is a measure $1$ class of sets whose many-one degrees
contain infinite antichains of one-one degrees. This class contains
all rigid many-one degrees. The present work generalises Batyrshin's
result to all nonrecursive and nonirreducible many-one degrees and
solves therefore Odifreddi's open problem.

The present work also proposes to deepen the study of reducibilities between
one-one and many-one in recursion theory in order to get a more complete
and detailed picture for the structures inside many-one degrees.
It namely proposes to study in more detail than before the
finite-one and bounded finite-one degrees. Odifreddi's Open Problem
is solved by showing that every nonrecursive finite-one degree which
does not coincide with the greatest one-one degree in a many-one degree
contains an infinite antichain of one-one degrees and furthermore allows
to embed any recursive partial order effectively into the structure of
one-one degrees inside the finite-one degree.
This is done by starting with a representative $A$ of
the finite-one degree and then constructing an array $B_0,B_1,B_2,\ldots$
of sets given by finite-one reductions to $A$ which are also all
one-one above $A$ and which form an antichain or embed a given
recursive partial order. In contrast to this, there are nonrecursive
bounded finite-one degrees consisting of a linearly ordered set of
one-one degrees without any incomparable pair of one-one degrees inside it.
Furthermore, some initial results about the structure of finite-one degrees
inside many-one degrees are obtained.
\end{abstract}

\medskip
\noindent
{\small {\bf Funding.}
Partial support to Frank Stephan and full support to Linus Richter
during the time of his past employment at the National University of
Singapore were provided by the Singapore Ministry of Education via
Academic Research Fund Tier 2 grant MOE-000538-01; Xiaoyan
Zhang started to work on the topic of this paper for an UROPS
(Under\-graduate Research Opportunity Programme in Science at NUS)
while he was an exchange student in Singapore and continued with
the work while completing the Masters at the Institute of Software
of the Chinese Academy of Sciences and he acknowledges an internship
at Workforce Optimizer Pte Ltd after enrolling
at the National University of Singapore as a PhD student.}

\section{Introduction}

\noindent
Post \cite{Po1944} investigated in his paper the classical reducibilities like
many-one, truth-table and Turing in order to determine which of them had
intermediate recursively enumerable degrees besides the recursive degree and
the degree of the halting problem; he answered it positively for all strong
degrees, but left it open for Turing reducibility. The strongest form of
reducibility are the many-one reducibility and their variants. Here $f$
reduces $A$ to $B$ iff for all $x$, $x \in A \Leftrightarrow f(x) \in B$;
the variants satisfy the additional request that $f$ is one-one or
finite-one. In all cases, for recursion theory, $f$ has to be a
recursive function. The work initiated by Post led to a comprehensive body
of research comparing and relating strong reducibilities and their degrees.
Odifreddi surveyed in an article \cite{Od1981} and in
his books ``Classical Recursion Theory''
\cite{Od1989,Od1999} also the structure of one-one degrees inside
many-one degrees; the first book concentrated on general degrees while
the second book specialised at limit-recursive and recursively enumerable
degrees.

Young \cite{Yo1966b} showed in 1966 that there are two base-cases
for many-one degrees, either they consist of a single one-one degree
or they consist of infinitely many one-one degrees. For example,
the many-one degree of the empty set or of the
full set of natural numbers $\mathbb N$ consist of a single set
and thus a single one-one degree. Myhill \cite{My1955} had shown
already in 1955 that the many-one degree of the halting problem consists
of sets which are pairwise equivalent by a recursive bijection \cite{My1955},
thus the many-one degree of the halting problem is a single one-one degree;
such many-one degrees are called irreducible.
On the other hand, simple sets as introduced by Post \cite{Po1944}
satisfy that their many-one degree consists of infinitely many one-one
degrees. Young \cite{Yo1966b} furthermore showed that the recursive
and nonrecursive nonirreducible many-one degrees differ. While the
greatest recursive many-one degree allows only to embed linear orders
isomorphic to either a finite ordering or to the natural numbers with their
default ordering or to the natural numbers with their default ordering plus
one element above them, every nonrecursive nonirreducible many-one degree
allows to embed any countable linear order into the structure of one-one
degrees which it contains. In particular the countable dense linear order is
embeddable into the one-one degrees inside any given nonirreducible and
nonrecursive many-one degree. Young worked also on other aspects of
many-one and one-one degrees \cite{Yo1964,Yo1966a}.

Rogers \cite{Ro1967} showed that every many-one degree contains a greatest
one-one degree and this one-one degrees consists of all the cylinders
in the many-one degree. Here, a cylinder is a set $A$ which is one-one
equivalent to the Cartesian product of $A$ with the set of the natural
numbers. Dekker and Myhill \cite{DM1960} showed that there are many-one
degrees without a least one-one degree, examples of these are the many-one
degrees of simple sets. Furthermore, the greatest recursive many-one degree
has two minimal one-one degrees, the singleton sets and their complements;
this degree indeed contains antichains of length $2$ but not of length $3$.
One of the corollaries to the results in the present work is that this
degree is also the only many-one degree with this property.
Motivated by Young's result on the embeddability of countable linear orders,
which did not include an embeddability result for countable partial orders,
D\"egtev \cite{De1976} showed that recursively enumerable nonrecursive and
nonirreducible many-one degrees can embed countable partial orders including
infinite antichains into the structure of one-one degrees inside them.

This motivated Odifreddi \cite[Open Problem 5]{Od1981} to ask explicitly
whether every nonrecursive and nonirreducible manyone degree contains an
infinite antichain of one-one degrees. Batyrshin \cite{Ba2021} confirmed
this for limit-recursive many-one degrees. Here a limit-recursive set
is a set Turing reducible to the halting problem $K$. The main result
of this paper is to generalise this result to all many-one degrees,
thus answering Odifreddi's question from 1981 affirmatively.

Furthermore, the present work analyses
the situation when taking the intermediate finite-one degrees between the
one-one degrees and many-one degrees into account. These have only be
rarely investigated; however, the authors and their overseas colleagues
were able to track the following references to finite-one reductions
in recursion theory and computational complexity.
Maslova \cite{Ma1979} studied recursively enumerable finite-one
degrees in 1979 and showed that every recursively enumerable many-one
degree which is neither recursive nor irreducible contains an antichain of
finite-one degrees; she also showed that inside a recursively enumerable
many-one degree, the simple sets form a least finite-one degree
in the case that they exist in the degree. The classes of simple,
hypersimple and effectively simple sets are closed under finite-one
reducibility. The three bounded versions of many-one reducibility
investigated by Maslova are the $bm$-reducibility (called finite-one
reducibility in this paper) and $bm'$-reducibility (where some partial
recursive function bounds the number of preimages mapped to a given image)
and $bm''$ reducibility where some recursive function computes for each $y$
how many preimages it has.
Maslova mentioned work of Gerhard Lischke in computational
complexity on variants of the finite-one reducibility, without
giving a precise reference in which works he published it; perhaps
he only discussed the work with Maslova and did not publish it.
Lischke's early work deals with ways to measure computational complexity
in a natural way and might therefore also have touched reducibilities,
see for example the following works of Lischke, one joint with Berger
\cite{BL1977,Li1975,Li1976,Li1977}.

In the 1990ies
the intermediate reducibilties captured the attention of Edith and Lane
Hemaspaandra and they published their findings in a note in
Theoretical Computer Science \cite{HH1994}. Bazhenov, Mustafa and
Ospichev \cite{BMO2019} carried over the study of
finite-one reducibility from a study of reducibility between sets
to a study of reducibilities between numberings of uniformly recursively
enumerable sets. The next paper using finite-one reducibility is that of
Kjos-Hanssen and Webb \cite{KW2021}; they used it as a tool to study
various forms of randomness and were not able to track down any of the
previous papers addressing the topic. Furthermore, Webb \cite{We2022}
also treats this reducibility in his PhD thesis. There have been various names
for finite-one reducibility in the literature: Maslova \cite{Ma1979}
called it ``bounded many-one reducibility'' and defined three versions
$bm$, $bm'$ and $bm''$ among which finite-one reducibility is the first
and the third just says that there is a recursive function which computes
for every $y$ how many $x$ are mapped by the reduction to $y$---for this,
their number has of course to be finite, but it might be $0$.
Bounded finite-one reducibility implies $bm$ and $bm'$ and is incomparable
to $bm''$. Infinite and coinfinite recursive sets are one-one equivalent
and $bm''$-equivalent. Bazhenov, Mustafa and Ospichev \cite{BMO2019}
studied finite-one reducibility and called it bounded reducibility for
a reducibility between numberings. Kjos-Hanssen and Webb \cite{KW2021,We2022}
called it finite-to-one reducibility. In the present paper, the authors
use the word ``finite-one reducibility'' which is parallel to the words
``one-one reducibility'' and ``many-one reducibility'' used by
Lachlan \cite{La1969}, Moschovakis \cite{Mo1966} and Young~\cite{Yo1966b}.
Furthermore, bounded finite-one reducibility of the present work could also
be called ``constant-bounded many-one reducibility'', if one follows
the terminology of Maslova \cite{Ma1979}.

The following is now known about finite-one degrees inside many-one
degrees: A finite-one degree containing
the greatest finite-one degree inside its many-one degree is irreducible
and consists exactly of this one-one degree \cite{Ma1979}.
All other nonrecursive finite-one degrees contain infinite antichains of
bounded finite-one degrees. Those in turn might or might not contain
antichains of one-one degrees. There are infinitely many bounded finite-one
degrees such that the structure of one-one degrees inside them is
order-isomorphic to the natural numbers with their default order. The collection
of recursive sets consists of five finite-one degrees which all consist
of a single bounded finite-one degree and out of which three are a single
one-one degree and two are ascending chains of one-one degrees.

The interested reader finds, besides the information
provided in the textbooks of Odi\-fred\-di \cite{Od1989,Od1999}
and his survey article \cite{Od1981},
also further background on strong reducibilities and recursion theory
in general in other recursion-theoretic textbooks like those of
Calude~\cite{Ca2002}, Chong and Yu~\cite{CY2015},
Downey and Hirschfeldt~\cite{DH2010}, Li and Vit\'anyi~\cite{LV2008},
Nies~\cite{Ni2009}, Rogers~\cite{Ro1967} and Soare~\cite{So1987}.

\section{One-one degrees inside finite-one degrees}

\noindent
The three main reducibilities studied in the present work are many-one,
one-one and finite-one; these type of reducibilities had been used
by Post \cite{Po1944}, Myhill \cite{My1955} and Kjos-Hanssen and
Webb \cite{KW2021}, respectively. Reducibilities are a core concept
of recursion theory and Post \cite{Po1944} studied them in order to get
initial results towards his question whether there are recursively enumerable
Turing degrees which are neither recursive nor complete (= Turing equivalent
to the halting problem). The distinction between one-one, finite-one
and many-one functions was already studied for centuries in other branches
of mathematics, with one-one and many-one being the most common types.
Finite-one functions were explicitly used in the construction of the
Rudin-Blass ordering between ultrafilters in set theory, see
Laflamme and Zhu \cite{LZ1998} for an important paper on that notion;
this reducibility is not really a finite-one reducibility itself, but it uses
finite-one functions between the spaces on which the ultrafilters
are build as part of its definition.
Sets $A,B,\ldots$ used in this paper are subsets of the natural numbers
${\mathbb N} = \{0,1,2,\ldots\}$ and sets are identified with their
characteristic functions, so if $x \in A$ then $A(x)=1$ else $A(x)=0$.

\begin{defn} \label{def-reduction}
A function $f$ is a many-one reduction from $A$ to $B$ iff $f$ is recursive
and for all $x$, $A(x) = B(f(x))$. Furthermore, a many-one reduction $f$
is a finite-one reduction if for each $y$ there are at most finitely many $x$
with $f(x) = y$ and a one-one reduction if for each $y$
there is at most one $x$ with $f(x) = y$. Between these two reductions
is the bounded finite-one reduction where there is a constant $c$ such that
for every $y$ there are at most $c$ numbers $x$ with $f(x) = y$. Furthermore,
let $X \oplus Y = \{2z: z \in X\} \cup \{2z+1: z \in Y\}$
be the join of $X$ and $Y$. Now $X,Y$ are both one-one reducible
to their join. The many-one degree of $A$ is the set of all $B$
such that first $B$ is many-one reducible to $A$ and second $A$
is many-one reducible to $B$; finite-one, bounded finite-one and
one-one degrees are defined analogously.
\end{defn}

\begin{theorem} \label{thm-one}
Let $A$ be an infinite and coinfinite subset of the natural numbers.
Then one can construct an array $B_0,B_1,\ldots$ of sets finite-one
equivalent to $A$ such that each of them is one-one above $A$ and
either the $B_1,B_2,\ldots$ forms an antichain in the one-one degrees
or $A$ is in the greatest one-one degree in its many-one degree and
all $B_k$ are one-one equivalent to $A$.
\end{theorem}

\begin{proof}
One defines each set $B_i$ via a surjective finite-one reduction $h_i$ from
natural numbers to natural numbers where $B_i(x) = A(h_i(x))$. By construction,
every $B_i$ is finite-one reducible to $A$ and, by surjectivity,
$A$ is one-one reducible to every $B_i$. The functions $h_i$ are constructed
in stages and implicitly define $B_i$ as indicated above; the main
constraint is that each $y$ has at least one and at most finitely
many preimages $x$. The overall quantity and positions of these preimages
is controlled by a finite injury construction.
Furthermore, one let $\varphi_1,\varphi_2,\ldots$ be a numbering of
all partial-recursive functions which are one-one on their domain
and which satisfy that whenever $\varphi_e(x)$ is undefined
so is $\varphi_e(x+1)$. Note that for each one-one reduction between
two sets, one can find a $\varphi_e$ in this list which coincides with
this one-one reduction. Now one defines the following
requirements $R_{e,i,j}$:
\begin{enumerate}
\item $R_{0,0,0}$: This requirement makes sure that for each
  $y,i$ there is an $x$ such that $h_i(x) = y$ and that each
  $h_i(x)$ is eventually defined.
\item $R_{e,i,j}$ with $e>0$ and $i \neq j$: This requirement
  makes sure that whenever $\varphi_e$ is a one-one reduction
  from $B_i$ to $B_j$ then a finite variant of a strictly increasing
  selfreduction from $A$ to $A$ is constructed - thus such a selfreduction
  exists and $A$ is a cylinder, that is, the finite-one degree of $A$
  consists of a single one-one degree which is the largest one-one degree
  in the many-one degree of $A$.
\end{enumerate}
For this, one assumes a default enumeration of all requirements with
$R_{0,0,0}$ coming first and $R_{e,i,j}$ does not come strictly before
$R_{e',i',j'}$ when $e' \leq e \wedge i' \leq i \wedge j' \leq j$.
Let $g(R_{e,i,j})$ be the natural number assigned to the requirement
$R_{e,i,j}$ and assume that $g$ is a recursive bijection. In other word,
if $d = g(R_{e,i,j})$ then $R_{e,i,j}$ is the $d$-th requirement with
$R_{0,0,0}$ being the zeroth requirement.

\begin{algo} \label{algo-one}
The algorithm to construct the array runs in stages.
Each stage takes so long until
all the requirements in it have done the actions linked to it
or are skipped due to not requiring any current action.

At stage $s$, one first satisfies the requirement $R_{0,0,0}$:

While there are $i,y \leq s$ such that there is no $x$
on which $h_i(x)$ is already defined to be $y$ then take
the first $x$ where $h_i(x)$ is not yet defined and let
$h_i(x) = y$. Furthermore, while there are $x,i \leq s$ with
$h_i(x)$ being undefined, let $h_i(x) = y$ for the first
$y$ not yet in the range of $h_i$.

Now, one looks for $d=1,2,\ldots,s$ at that requirement
$R_{e,i,j}$ which satisfies $g(R_{e,i,j}) = d$:

If the marker $m_{e,i,j}$ is currently not sitting on some
number then do the following: Define the set $D = \{y: d \leq y \leq s$ and
$f_d(y)$ is not already defined and none of the markers for
requirements $R_{e',i',j'}$ with $0 < g(R_{e',i',j'}) < d$
is sitting on $y\}$; if $D$ is not empty then place $m_{e,i,j}$ on
the minimum of $D$ and remove all other markers sitting there
(they have lower priority).

If now the marker is sitting on a number $y$ and $f_d(y)$
is not yet defined then one does the following steps.

While the number of $x$ with $h_i(x) = y$ is less or equal to
the number of $x'$ with $h_j(x') \leq y$ do begin select the least
$x$ where $h_i(x)$ is not yet defined and let $h_i(x) = y$ end.

If $\varphi_e(x)$ is defined on all $x$ where currently $h_i(x)=y$
within $s$ steps then one finds an $x$ with $h_i(x) = y$ and either
$h_j(\varphi_e(x))$ is not yet defined or $h_j(\varphi_e(x)) > y$;
fix this $x$ for now.
If $h_j(\varphi_e(x))$ is not yet defined then one defines
$h_j(\varphi_e(x)) = y'$ for the first $y'>y$ which is not
yet in the range of $h_j$, so that from now on the second
subcase $h_j(\varphi_e(x)) > y$ holds.
Now one defines that $f_d(y) = h_j(\varphi_e(x))$.

Now one concludes the step by removing the marker $m_{e,i,j}$
from its position if $f_d(y)$ has been defined. The marker remains
on its position if the marker is waiting for future stages
until all $\varphi_e(x)$ with $h_i(x) = y$ become defined.

End of activity for $R_{e,i,j}$ inside stage $s$.

Once all requirements $R_{e,i,j}$ with $g(R_{e,i,j}) \leq d$ are handled,
this is the end of stage $s$ and the algorithm goes to stage $s+1$.
\end{algo}

\noindent
Recall that $B_i(x) = A(h_i(x))$ for all $i,x$. Thus
every $B_i$ is many-one equivalent to $A$ via $h_i$.
Theorem~\ref{thm-one} will now be proven by a series of
four claims showing important properties of the array
of sets constructed with Algorithm~\ref{algo-one};
note that $A$ is a parameter, but it is only evaluated
at one point: To make a finite modification of some function
$f_d$ to show that $A$ is a cylinder in the case that the
$B_i$ do not form an antichain of one-one degrees inside
the finite-one degree of $A$.

\begin{claim}
At every stage, all $h_i$ have a finite domain and overall only
finitely many new values of the functions $f_i$ are defined in
a stage.
\end{claim}

\noindent
To see this claim,
assume by way of contradiction that this would be false and let
$s$ be the first stage where for infinitely many pairs a value
$h_i(x)$ is newly defined. Furthermore, there  must be a first
requirement with respect to the requirement number $d = g(R_{e,i,j})$
for which this happens. It cannot be that $d=0$ as that requirement
defines at most for each $(i,y)$ with $i,y \leq s$ one value
$h_i(x) = y$ and furthermore for each $(i,x)$ with $i,x \leq s$
at most one further value.
This are at most $2 \cdot (s+1)^2$ many definitions and thus only
finitely many. For $d>0$, there are only
two definition steps inside the activity. The substep
\begin{quote}
``While the number of $x$ with $h_i(x) = y$ is less or equal to
the number of $x'$ with $h_j(x') \leq y$ do begin select the least
$x$ where $h_i(x)$ is not yet defined and let $h_i(x) = y$ end.''
\end{quote}
defines only finitely many values $h_i(x)$ as by assumption there
are only finitely many value $h_j(x)$ defined before. Furthermore,
there is in this step at most one further point where a value is
defined and that is defining $h_j(x)=y'$ for some value $y'>s$.
Thus in contrary to the assumption, only finitely many new
definitions are done in this step. Therefore also at most
finitely many new definitions are done in stage $s$.

\begin{claim}
For all $i,y$ there are at least one and at most finitely
many $x$ with $h_i(x) = y$. In particular each $B_i$ is finite-one
equivalent to $A$ and $A$ is one-one reducible to each $B_i$.
\end{claim}

\noindent
Note that $R_{0,0,0}$ enforces that for each pair $(i,y)$ there
is at least one $x$ with $h_i(x) = y$ and that therefore all
functions $h_i$ are surjective. Thus one has only to show
that the overall number is finite. By the first verification item,
this can only happen if some marker sits forever on a value $y$
for a requirement $R_{e,i,j}$. So let $y$ be the least number
so that there are infinitely many $x$ with $h_i(x) = y$ for one
single $i$. Furthermore, this can only be caused by finitely
many requirements, as requirements $R_{e,i,j}$ with $g(R_{e,i,j})>y$
do not define $h_i$ on any $x$ to have the value $y$. In addition
all the requirements together who act at markers strictly below
$y$ can only define $h_i(x)$ for the first $x$ which takes the
value $y$ but not for further ones, thus these markers cannot
cause the problem. So it must be one requirement $R_{e,i,j}$
whose marker settles on $y$ forever. Let $s$ be so large that
all requirements which define some value $h_j(x')=y'$ for
some $y'<y$ have already done so, by choice of $y$ there are
only finitely many; furthermore, $s$ is so large such that there
is at least one $x$ with $h_j(x) = 0$ already being defined.
Furthermore the marker
$m_{e,i,j}$ sits from stage $s$ onwards forever on $y$, and the
markers of the higher priority requirements are either sitting
on values strictly above $y$ forever or they have converged
to a lower value and will not be moved again. Thus, no other marker
makes new definitions of the form $h_j(x') = y'$ for some $y' \leq y$
from now on. Let $t$ be the number of $x$ with $h_j(x) \leq y$,
this number is thus constant. It follows by the way that the
activity related to requirement $R_{e,i,j}$ is defined that
there are at most $t+1$ many values $x$ for which the
requirement defines $h_i(x) = y$, thus there are only finitely
many and not infinitely many. So the assumption was false
and it follows from contraposition that each $B_i$ is indeed
finite-one equivalent to $A$ via the mapping $h_i$.

\begin{claim} \label{clm:six}
Assume that $\varphi_e$ one-one reduces $B_i$ to $B_j$ and $d=g(R_{e,i,j})$.
Now $f_d$ is defined on almost all inputs and for all $y$ in its domain,
$f_d$ satisfies $A(y) = A(f_d(y))$ and $f_d(y)>y$.
\end{claim}

\noindent
To see this claim, consider any $y$ in the domain of $f_d$.
Then $f_d(y) = h_j(\varphi_e(x))$ for some $x$ with $h_i(x) = y$
and $h_j(\varphi_e(x)) > y$. Thus $f_d(y)>y$. Furthermore,
$B_i(x) = A(h_i(x))$ by definition and $B_j(\varphi_e(x)) = A(y)$
by the assumed correctness of $\varphi_e$ and the equality-chain
$A(f_d(y)) = A(h_j(\varphi_e(x))) = B_j(\varphi_e(x)) = B_i(x) = A(y)$
shows that $A$ is a partial selfreduction on $A$.

Now assume that $y \notin dom(f_d)$ and $y \geq d$. There are only
finitely many values on which a marker of a higher priority requirement
stays forever, so assume that $y$ is not one of them. One possibility
is now that the marker $m_{e,i,j}$ itself stays forever on $y$.
Then there is some stage $s$ large enough such that all
$x$ with $h_i(x) \leq y$ or $h_j(x) \leq y$ are already defined
and $\varphi_e(x)$ is defined for all $x$ with $h_i(x) = y$
and $m_{e,i,j}$ is sitting on $y$. Then it must be that there
are more $x$ with $h_i(x) = y$ than $x$ with $h_j(x) \leq y$,
otherwise $h_i(x) = y$ would become defined for some further $x$ in
the future; as this does not happen, this must already have been
defined before. Now there must be an $x$ such that $h_i(x) = y$
and $h_j(\varphi_e(x))$ is either undefined or strictly above $y$;
however, this would force in this stage the value of $f_d(y)$
to become defined. As by assumption this does not happen,
it cannot be that the marker $m_{e,i,j}$ sits on some $y$
forever. Furthermore, it cannot be that the $y$ is overlooked,
that is, from some time point onwards, for all further stages,
none of the markers associated to a requirement $R_{e',i',j'}$
with $g(R_{e',i',j'}) \leq d$ sits on $y$. Without loss of generality,
$y$ is the least one of the numbers greater equal to $d$ for which
$f_d(y)$ is undefined and to which none of the markers $m_{e',i',j'}$
with $g(R_{e',i',j'}) \leq d$ converges. Then $y$ would for almost
all stages qualify as the value on which $m_{e,i,j}$ will take
and thus, in contrary to the assumption, $f_d(y)$ gets defined.
This completes the proof that the domain of $f_d$ is cofinite.

\begin{claim}
If $\varphi_e$ is a one-one reduction from $B_i$ to $B_j$
then $A$ is a cylinder, that is, in the greatest one-one degree
of its many-one degree. Furthermore, if $A$ is infinite, coinfinite
and not a cylinder, then $B_0,B_1,\ldots$ is an antichain
inside the finite-one degree of $A$.
\end{claim}

\noindent
By the above, $f_d$ is defined for almost all $y$. One makes it
total by mapping all remaining $y$ to the first $y'>y$ with
$A(y') = A(y)$. Now let $f_{d,0}(x) = x$ and
$f_{d,k+1}(x) = f_d(f_{d,k}(x))$ for all $k$. Using
the functions $f_{d,k}$ one can, given a
many-one reduction $f'$ from some set $C$ to $A$, one can
also obtain a one-one reduction $f''$ by defining
for $x=0,1,\ldots$ the value $f''(x) = f_{d,k}(f'(x))$ for the
first $k$ where $f_{d,k}(f'(x)) \notin \{f''(x'): x' < x\}$.
This proves that $C$ is one-one reducible to $A$ and thus $A$
is in the greatest one-one degree of its many-one degree.

If now $A$ is infinite and coinfinite and not a cylinder,
then no $f_d$ can be finitely extended to a total strictly
increasing selfreduction of $A$, thus either $f_d$ is not
a partial selfreduction or its domain is coinfinite. This
happens only if the corresponding $\varphi_e$ is not a
one-one reduction from $B_i$ to $B_j$. Thus for the array
of the $B_i$ constructed, there exist no distinct indices $i,j$
and no one-one reduction $\varphi_e$ such that $\varphi_e$
reduces $B_i$ to $B_j$. So $B_0,B_1,\ldots$ is an
antichain in the one-one degrees, the property that all
$B_i$ are strictly one-one above $A$ but still finite-one
equivalent to $A$ follows from the construction.
\end{proof}

\noindent
Theorem~\ref{thm-one} can be improved to the following Theorem~\ref{thm-two}.
For this, let $\sqsubset$ be a recursive partial order, that is, a relation
which is transitive and which satisfies for all distinct $i,j$ that either
(a) $i \sqsubset j$ or (b) $j \sqsubset i$ or (c) $i,j$ are incomparable;
note that the case $i \sqsubset j \wedge j \sqsubset i$ does not
occur. An algorithm can compute for each pair of distinct $i,j$
which of the above three cases (a), (b) and (c) applies.

\begin{theorem} \label{thm-two}
Let $\sqsubset$ be any recursive partial order.
Let $A$ be an infinite and coinfinite set which is not in the greatest
one-one degree of its many-one degree. Then one can construct $B_0,B_1,
\ldots$ finite-one equivalent to $A$ such that $B_i$ is one-one reducible
to $B_j$ if and only if either $i=j$ or $i \sqsubset j$; that is,
one can embed every recursive partial order into the one-one degrees
inside the finite-one degree of $A$.
\end{theorem}

\begin{proof}
The proof of Theorem~\ref{thm-two} is similar to that of
Theorem~\ref{thm-one}. Thus it is mainly listed out what changes
are to be done to prove the result. Let $\sqsubseteq$ be a recursive
preorder, that is, it is transitive and reflexive.
Again one defines each set $B_i$ via a surjective finite-one reduction $h_i$
from natural numbers to natural numbers where $B_i(x) = A(h_i(x))$. By
construction, every $B_i$ is finite-one reducible to $A$ and, by surjectivity,
$A$ is one-one reducible to every $B_i$. The functions $h_i$ are constructed
in stages and implicitly define $B_i$ as indicated above; the main
constraint is that each $y$ has at least one and at most finitely
many preimages $x$. The overall quantity and positions of these preimages
is controlled by a finite injury construction.
Furthermore, one let $\varphi_1,\varphi_2,\ldots$ be an acceptable numbering
of all partial-recursive one-one functions with $\varphi_e(x+1)$ only being
defined when $\varphi_e(x)$ is and one considers the following
requirements $R_{e,i,j}$:
\begin{enumerate}
\item $R_{0,0,0}$: This requirement makes sure that for each
  $y,i$ there is an $x$ such that $h_i(x) = y$ and that each
  $h_i(x)$ is eventually defined.
\item $R_{0,i,j}$ with $i \neq j$ and $j \sqsubset i$: This
  requirement makes sure that, for almost all $y$, there are at
  strictly more $x$ with $h_i(x)=y$ as there are $x'$ with
  $h_j(x') \leq y$. Thus, for all but finitely many $y$, one
  can map the $x'$ with $h_j(x') = y$ in a one-one way to the
  $x$ with $h_i(x) = y$ and one can use $A'$ to map the remaining
  finitely many $x'$ to counterparts $x$ with $B_i(x) = B_j(x')$.
\item $R_{e,i,j}$ with $e>0$ and $i \neq j$ and $i \not\sqsubset j$:
  This requirement makes sure that whenever $\varphi_e$ is a one-one reduction
  from $B_i$ to $B_j$ then a finite variant of a strictly increasing
  selfreduction from $A$ to $A$ is constructed---thus, such a selfreduction
  exists and $A$ is a cylinder, that is, the finite-one degree of $A$
  consists of a single one-one degree which is the largest one-one degree
  in the many-one degree of $A$.
\end{enumerate}
For this, one assumes a default enumeration of all requirements with
$R_{0,0,0}$ coming first and $R_{e,i,j}$ does not come strictly before
$R_{e',i',j'}$ when $e' \leq e \wedge i' \leq i \wedge j' \leq j$.
Let $g(R_{e,i,j})$ be the natural number assigned to the requirement
$R_{e,i,j}$ and assume that $g$ is a recursive bijection.

\begin{algo} \label{algo-two}
The algorithm runs in stages. Each stage takes so long until
all the requirements in it have done the actions linked to it.

At stage $s$, one first satisfies the requirement $R_{0,0,0}$:

While there are $i,y \leq s$ such that there is no $x$
on which $h_i(x)$ is already defined to be $y$ then take
the first $x$ where $h_i(x)$ is not yet defined and let
$h_i(x) = y$. Furthermore, while there are $x,i \leq s$ with
$h_i(x)$ being undefined, let $h_i(x) = y$ for the first
$y$ not yet in the range of $h_i$.

Now, one looks for $d=1,2,\ldots,s$ at the requirement
$R_{e,i,j}$ with $g(R_{e,i,j}) = d$:

If the marker $m_{e,i,j}$ is currently not sitting on some
number then do the following: Define the set $D = \{y: d \leq y \leq s$
and (either $f_d(y)$ is not already defined or $e=0$ and there are
at least as many $x'$ with $h_j(x') \leq y$ as $x$ with $h_i(x) = y$)
and none of the markers for
requirements $R_{e',i',j'}$ with $0 < g(R_{e',i',j'}) < d$
is sitting on $y\}$; if $D$ is not empty then place $m_{e,i,j}$ on
the minimum of $D$ and, if $e>0$, then remove all other markers sitting
there (they have lower priority).

If now the marker $m_{e,i,j}$ is sitting on a number $y$ then
one does the following steps.

While the number of $x$ with $h_i(x) = y$ is less or equal to
the number of $x'$ with $h_j(x') \leq y$ do begin select the least
$x$ where $h_i(x)$ is not yet defined and let $h_i(x) = y$ end.

If $e>0$ and $\varphi_e(x)$ is defined on all $x$ where currently $h_i(x)=y$
within $s$ steps then one finds an $x$ with $h_i(x) = y$ and either
$h_j(\varphi_e(x))$ is not yet defined or $h_j(\varphi_e(x)) > y$;
fix this $x$ for now.

If $h_j(\varphi_e(x))$ is not yet defined then one defines
$h_j(\varphi_e(x)) = y'$ for the first $y'>y$ which is not
yet in the range of $h_j$, so that from now on the second
subcase $h_j(\varphi_e(x)) > y$ holds.
Now one defines that $f_d(y) = h_j(\varphi_e(x))$.

Now one concludes the step by removing the marker $m_{e,i,j}$
from its position if either $e=0$ or $f_d(y)$ has been defined.
It remains on its position if (a) the marker is waiting
for future stages until all $\varphi_e(x)$
with $h_i(x) = y$ become defined and (b) no higher priority
marker with $e>0$ goes onto $y$.

End of activity for $R_{e,i,j}$ inside stage $s$.

Once all requirements $R_{e,i,j}$ with $g(R_{e,i,j}) \leq d$ are handled,
this is the end of stage $s$ and the algorithm goes to stage $s+1$.
\end{algo}

\noindent
Recall that $B_i(x) = A(h_i(x))$ for all $i,x$. Thus
every $B_i$ is many-one equivalent to $A$ via $h_i$.
Again, one will establish the properties of the array of
sets constructed by above algorithm~\ref{algo-two} through
a series of four claims.

\begin{claim} At every stage, all $h_i$ have a finite domain and
overall only finitely many new values of the functions $f_i$
are defined in a stage.
\end{claim}

\noindent
Assume by way of contradiction that this claim would be false and let
$s$ be the first stage where for infinitely many pairs a value
$h_i(x)$ is newly defined. Furthermore, there  must be a first
requirement with respect to the requirement number $d = g(R_{e,i,j})$
for which this happens. It cannot be that $d=0$ as that requirement
defines at most for each $(i,y)$ with $i,y \leq s$ one value
$h_i(x) = y$ and furthermore for each $(i,x)$ with $i,x \leq s$
at most one further value.
This are at most $2 \cdot (s+1)^2$ many definitions and thus only
finitely many. For $d>0$, there are only
two definition steps inside the activity. The substep
\begin{quote}
  ``While the number of $x$ with $h_i(x) = y$ is less or equal to
  the number of $x'$ with $h_j(x') \leq y$ do begin select the least
  $x$ where $h_i(x)$ is not yet defined and let $h_i(x) = y$ end.''
\end{quote}
defines only finitely many values $h_i(x)$ as by assumption there
are only finitely many value $h_j(x)$ defined before. Furthermore,
there is in this step at most one further point where a value is
defined and that is defining $h_j(x)=y'$ for some value $y'>y$.
Thus in contrary to the assumption, only finitely many new
definitions are done in this step. Therefore also at most
finitely many new definitions are done in stage $s$.

\begin{claim} \label{cl:claimtwo}
For all $i,y$ there are at least one and at most finitely
many $x$ with $h_i(x) = y$. In particular each $B_i$ is finite-one
equivalent to $A$ and $A$ is one-one reducible to each $B_i$.
\end{claim}

\noindent
Note that $R_{0,0,0}$ enforces that for each pair $(i,y)$ there
is at least one $x$ with $h_i(x) = y$. Thus one has only to show
that the overall number is finite.

So assume that there are pairs $(i,y)$ with infinitely many $x$
satisfying $h_i(x) = y$. Among those pairs, take $y$ as small as possible
and fix it from now on.

There are at most $y$ indices $i$ for which there are at least
two $x$ with $h_i(x) = y$. The reason is that for each
such $i$ there must be a requirement $R_{e,i,j}$ where $e=0$ is
possible with $j \neq i$ and $1 \leq g(R_{e,i,j}) \leq y$. Note that
except for the first $x$ with $h_i(x) = y$, all further ones are
defined by some requirement $R_{e,i,j}$ with $g(R_{e,i,j}) \leq y$
and $i$ must be the second parameter, not the third of the requirement.

Now let $E = \{i,j:$ there exists requirement $R_{e,i,j}$ with $g(R_{e,i,j})
\leq y$ and $j \neq i\}$, let $E' = \{i \in E:$ there are infinitely
many $x$ with $h_i(x) = y\}$ and let $E'' = \{i \in E'$: no $j \sqsubset i$
is in $E'\}$. Note that $E,E',E''$ are finite sets and as $\sqsubset$ is a
partial order, $E' \neq \emptyset$ implies $E'' \neq \emptyset$ and
that only the $i \in E$ satisfy that there are for some $y' \leq y$ at
least two $x'$ with $h_i(x') = y'$.
For each $i \in E''$ there must be be a marker $R_{e,i,j}$ with
$g(R_{e,i,j}) \leq y$ sitting on $y$ infinitely often in order to
achieve that $i \in E'$, that is, that there are infinitely many
$x$ with $h_i(x) = y$. Therefore one has to look at the stages
and marker movement in more detail.

So let $s$ be so large that the following holds:
\begin{enumerate}
\item $s \geq y$.
\item If $i \in E$ and $f_d(y)$ gets eventually defined then this
   happened before stage $s$.
\item All requirements which define only finitely many $h_i(x)$
   with $h_i(x) \leq y$ and $i \in E$ have done this before stage $s$.
\item All markers which go only finitely often onto a number $y' \leq y$
   have completed these actions before stage $s$.
\item All requirements $R_{e,i,j}$ with $g(R_{e,i,j}) \leq y$ which
   have a marker only finitely often sitting on some $y' \leq y$
   have removed this marker forever from the corresponding $y'$
   before stage~$s$.
\end{enumerate}
Now let $F = \{$Requirement $R_{e',i',j'}$: $m_{e',i',j'}$ is on $y$
for infinitely many stages and $e'>0\}$ and let $R_{e,i,j}$ be the
member of $F$ where $d = g(R_{e,i,j})$ is minimal. Then $f_d(e)$
must remain undefined forever and therefore the marker will not
be released; the only higher priority markers which may get attention
are those where $e=0$ and those do not remove $m_{e,i,j}$ from its
current position; note that by conditions 2 and 3 above in the choice
of $s$, markers which get attention only finitely often do this before
stage $s$ and will not be requesting it again at stage $s$ or later.
Therefore the only way to assign a new $h_{i'}(x') = y$
for an $i' \in E$ is when either $i'=i$ and the number of $x''$
with $h_j(x'') \leq y$ has increased before or when there is a
$i'' \sqsubset i'$ where $i'' \in E$
and the number of $x''$ with $h_{i''}(x'') \leq y$
has increased after stage $s$.
However, this requires that $i=i'$ or $i \sqsubset i'$
by the fact that $m_{e,i,j}$ is sitting on $y$ and does not make
space for other markers. Furthermore, $j \neq i$ and $j \not\sqsubset i$,
thus there are no new $x'$ with $h_{i'}(x) = y$ and
$i' = j \vee i' \sqsubset j \vee i' \sqsubset i$. Therefore
no new $x$ with $h_i(x) = i$ are added and,
in contrary to the assumption, $i \notin E'$. Thus $E' = \emptyset$,
that is, there is no $i$ with $h_i(x) = y$ for infinitely many $x$.
It follows that the stament of Claim~\ref{cl:claimtwo} is correct.

\begin{claim}
Assume that $e>0$ and $i \neq j$ and $\varphi_e$ is total and
one-one reduces $B_i$ to $B_j$ and $i \not\sqsubset j$.
Now the requirement $R_{e,i,j}$ exists and has some value $d$ and
$f_d$ is defined on almost all inputs and for all $y$ in its domain,
$f_d$ satisfies $A(y) = A(f_d(y))$ and $f_d(y)>y$.
\end{claim}

\noindent
This claim has the same proof as Claim~\ref{clm:six}.

\begin{claim}
If the two properties $\varphi_e$ is a one-one reduction from $B_i$ to $B_j$
and $i \not\sqsubseteq j$ jointly hold then $A$ is a cylinder,
that is, in the greatest one-one degree of its many-one degree.
Furthermore, if $i \sqsubseteq j$ then $B_i$ is one-one
reducible to $B_j$, independent on what set $A$ is, only provided
that $A$ is infinite and coinfinite. Thus,
if $A$ is neither recursive nor a cylinder, then $B_0,B_1,\ldots$ represent
an array of one-one degrees inside the finite-one degree of $A$, whose
ordering by one-one reducibility coincides with the partial order
$\sqsubseteq$ when made reflexive by using $\sqsubseteq$ instead of
$\sqsubset$ itself.
\end{claim}

\noindent
By the above, if $i \not\sqsubseteq j$ and $\varphi_e$ is total
then $f_d$ is defined for almost all $y$. One makes it
total by mapping all remaining $y$ to the first $y'>y$ with
$A(y') = A(y)$. Now let $f_{d,0}(x) = x$ and
$f_{d,k+1}(x) = f_d(f_{d,k}(x))$ for all $k$. Using
the functions $f_{d,k}$ one can, given a
many-one reduction $f'$ from some set $C$ to $A$, one can
also obtain a one-one reduction $f''$ by defining
for $x=0,1,\ldots$ the value $f''(x) = f_{d,k}(f'(x))$ for the
first $k$ where $f_{d,k}(f'(x)) \notin \{f''(x'): x' < x\}$.
This proves that $C$ is one-one reducible to $A$ and thus $A$
is in the greatest one-one degree of its many-one degree.

If now $A$ is infinite and coinfinite and not a cylinder
and $d = g(R_{e,i,j})$ for a requirement with $e>0$,
then $f_d$ cannot be finitely extended to a total strictly
increasing selfreduction of $A$, thus either $f_d$ is not
a partial selfreduction or its domain is coinfinite. This
happens only if the corresponding $\varphi_e$ is not a
one-one reduction from $B_i$ to $B_j$. As $e$ was chosen
arbitrarily, $B_i$ is not one-one reducible to $B_j$.

Furthermore, if $j \sqsubset i$, then $j \neq i$ and for all $d$ and
all $y \geq d$ with $d = g(R_{0,i,j})$, it holds that either
some higher priority requirement with a number
strictly below $g(0,i,j)$ has a marker eventually sitting forever
on $y$ or that there are strictly more $x$ with $h_i(x)=y$
than $x'$ with $h_j(x') = y$. Thus one can map the $x$ with $h_i(x)=y$
for almost all $y$ in a one-one way the $x'$ with $h_j(x') = y$
to $x$ with $h_i(x) = y$,
let $f$ be the so far constructed mapping. Now the
remaining finitely many undefined places of $f$ can be patched,
as for almost all $y$ there is an $x$ with $h_i(x) = y$ not in the
range of the $f$ constructed so far and while there are only finitely many
$x'$ not yet mapped to an $x$; thus $B_j$ is one-one reducible to $B_i$.
The just mentioned patching can be done using the oracle $A'$ and as
that oracle is used only finitely often, a nonuniformly obtained finite
table can replace it. If $j = i$ then the identity one-one reduces
$B_j$ to $B_i$, thus all cases of $j \sqsubseteq i$ are covered.
\end{proof}

\noindent
Mostowski \cite{Mo1938} has proven that there is a universal recursive
partial order, that is, a recursive partial order $\sqsubset$ such that
every further countable partial order can, though not effectively, be embedded
into it. Thus one has the below corollary, where the first part
follows directly from Theorem~\ref{thm-two} and the second part is
a direct consequence of the fact that when $A$ has a nonirreducible
finite-one degrees then $A$ has also a nonirreducible many-one degree.
Gu~\cite[Lemma 3.4]{Gu2023} provides an outline for this, quite
short, construction. Sacks \cite{Sa1963b} proved the related result
that every countable partial order can be embedded into the structure
of all Turing degrees with Turing reducibility as partial order.

\begin{corollary} \label{cor-for-thm-two}
Let $A$ be a nonrecursive set which is not isomorphic to a cylinder.
Now every at most countable partial order $(P,\sqsubset)$
can be embedded, in a noneffective way, into the following structures:
\begin{description}
\item[(a)] the partially ordered set of one-one degrees inside the
  finite-one degree of $A$;
\item[(b)] the partially ordered set of one-one degrees inside the
  many-one degree of $A$.
\end{description}
Furthermore, every nonrecursive nonirreducible many-one (finite-one)
degree has a representative $A$ which is neither recursive nor a
cylinder, thus every countable partial order can be embedded into
the structure of one-one degrees inside such a degree.
\end{corollary}

\begin{rem} \label{rem-for-thm-two}
This corollary has a direct consequence: A nonirreducible many-one
degree must consist of several finite-one degrees, for example the
greatest recursive many-one degree consists of three finite-one
degrees. The irreducible many-one degrees consist, in contrast, just
of one finite-one degrees. Odifreddi \cite[Problem 4]{Od1981} asks
for providing explicit criteria of either a many-one degree or of its
representatives such that the many-one degree is irreducible, these
criteria should either for the structure of the degree or for one or all
representing sets. An example for such a criterion is that all sets in
the many-one degree are cylinders; this follows from Myhill's Isomorphism
Theorem \cite{My1955} which states that if two sets are in the same one-one
degree then they are equivalent by a recursive permutation and therefore, if
one of them is of the form $A \times {\mathbb N}$ then the other one can also
be viewed as a set of pairs with an adjusted pairing function---the
adjustment stems from the bijection. Otherwise one could just use
Cantor's original function \cite{Ca1874}. So for Odifreddi's Problem 4,
one of the criteria that a many-one degree is irreducible is that
it coincides with a finite-one degree. Another criterion is that
the one-one degrees inside a given many-one degree form a chain---this
chain has then to collapse to just one single one-one degree. However,
the absence of antichains of length three is then
only equivalent to the following property: either the many-one degree is
irreducible or it coincides with the greatest recursive many-one degree.
This property is then also equivalent to the statement that every finite-one
degree inside the many-one degree coincides with a bounded finite-one degree,
see the sections below for more information about bounded finite-one degrees.
\end{rem}

\section{Finite-one degrees inside many-one degrees}

\noindent
The knowledge of the structure of finite-one degrees inside a many-one
degree is a bit limited compared what one knows about the structure
of one-one degrees inside a many-one degree. However, there are some
differences one can easily see when comparing the structure of one-one
degrees inside finite-one degrees with the structure of finite-one
degrees inside many-one degrees. The next proposition summarises
facts about finite-one degrees, which can be easily proven with
general knowledge about recursion theory. For these, $c(x,y)$ is
Cantor's pairing function \cite{Ca1874}
given as $c(x,y) = (x+y) \cdot (x+y+1)+y$.

\begin{theorem} \label{thm-finiteone}
The following properties hold for the structure of
finite-one degrees inside many-one degrees.
\begin{enumerate}
\item
The join $A \oplus B$ of two sets $A,B$ is the least upper bound
of $A$ and $B$ in the finite-one degrees and thus the finite-one
degrees form an upper semilattice; the same applies to the structure
of finite-one degrees inside many-one degrees.
\item
The greatest finite-one degree inside a many-one degree is always
irreducible, that is, coincides with a one-one degree and consists
only of cylinders.
\item
Every maximal set represents a minimal degree within the nonrecursive
finite-one degrees; however, due to maximal sets being simple, there are only
two of the five recursive finite-one degrees below them: The finite-one
degree of the cofinite sets with at least one nonelement and
the finite-one degree consisting of the single set $\mathbb N$.
\item
Every bi-immune set $A$ represents the least finite-one degree inside
its many-one degree and that many-one degree contains at least two
finite-one degrees. Furthermore, the many-one degree of $A$ does
neither have a least one-one degree nor a minimal one-one degree
and no recursive set is finite-one reducible to $A$.
\end{enumerate}
\end{theorem}

\begin{proof}
For the first item, consider $A,B$ and assume that $A,B$ are both
finite-one reducible to a set $C$ via $f,g$. Then for each $y$ there are only
finitely many $v$ with $f(v) = y$ and finitely many $w$ with
$g(w) = y$. Now let $h(2v) = f(v)$ and $h(2w+1) = g(w)$. Now,
for all $x$, $(A \oplus B)(x) = C(h(x))$, hence
$h$ is a many-one reduction from $A \oplus B$ to $C$.
Furthermore, there are, for each $y$, only finitely many $x$
with $h(x) = y$, as the number of these $x$
is the sum of the number of all $v$ with $f(v)=y$ and the
number of all $w$ with $g(w) = y$. Thus $h$ is a finite-one
reduction witnessing that $A \oplus B$ is finite-one
reducible to $C$. Furthermore, it is clear that both $A,B$ are
finite-one reducible to $A \oplus B$. Thus $A \oplus B$ represents
the least upper bound of $A$ and $B$ in the finite-one degrees.

For the second item, just note that if $A \times {\mathbb N}$ represents
the largest one-one degree inside a many-one degree and
if $A \times {\mathbb N}$ is finite-one
reducible to $B$ in the same many-one degree via $f$ then one can make
$f$ to be one-one as follows: Given $f$, one constructs $g$ inductively
over an enumeration of all pairs $c(x,y)$. If $f(c(x,y))$ is not yet in
the range of $g$ then one defines $g(c(x,y)) = f(c(x,y))$ else one
finds the first $z$ with $f(c(x,z))$ not in the so far defined range
of $g$ and let $g(c(x,y)) = f(c(x,z))$. The verification that this gives
a one-one reduction which is correct is left to the reader.

For the third item, let $A$ be a maximal set and assume that
$B$ is many-one reducible to $A$ via $f$; $B$ is therefore a
recursively enumerable set. Either finitely many
$y \notin A$ are in the range of $f$ and $B$ is recursive
or all but finitely many of the $y \notin A$ are in the range
of $f$ and $B$ is in the same many-one degree as $A$. Assume
the second one, as only that case is interesting; now one
constructs a finite-one reduction $g$ from $A$ to $B$.
$g(y)$ is defined to be the $x$ in that case which applies first;
the first case applies only for finitely many $y$ and they can be
tabled up in the algorithm; which of the cases two or three strikes
first in the remaining cases is determined by parallel search.
\begin{enumerate}
\item $y$ is neither in the range of $f$ nor in the set $A$ and
   $x$ is the smallest number not in $B$;
\item The number $x$ is the first $x$ found with $f(x) = y$;
\item $y$ is enumerated into $A$ and $x$ is the $y$-th element
   to be enumerated into $B$ by some fixed one-one enumeration of $B$;
   note that $B$ is infinite and every infinite recursively enumerable
   set has a recursive one-one enumeration.
\end{enumerate}

\medskip
\noindent
Independently of which case first holds, let $g(y)$ be the so found $x$
in the corresponding case. The function is finite-one for the following
reasons: The $x$ in the first case can only be chosen by the finitely
many $y$ which are neither in the range of $f$ nor in $A$ as well as
by one $y$ with $f(x) = y$. The second case contributes for each $x$
at most one $y$ with $g(y) = x$, this is the $y$ with $f(x) = y$.
The third case produces also, for each $x$, at most one $y$ with $g(y) = x$,
as that $x$ is the $y$-th element in a fixed recursive one-one enumeration
of $B$. Thus, except for the smallest $x$ not in $B$, each element in the
range of $g$ is only the image of at most two numbers.

For the fourth item, let $A$ be a bi-immune set, that is, a set
$A$ such that neither $A$ nor its complement has an infinite
recursive subset. Furthermore, assume that $A$ is many-one
reducible to some set $B$ via $f$. Then for each $y$, the set
$\{x: f(x) = y\}$ is a recursive set which is either a subset of
$A$ or its complement; thus it is finite. Therefore $f$ is a
finite-one reduction to $B$. As $A$ is bi-immune, $A$ is not a
cylinder and therefore there are at least two finite-one degrees
in the many-one degree of $A$. Furthermore, the finite-one
degree of $A$ does not have a least one-one degree, as given $A$ and
$B = \{x: x+1 \in A\}$, the set $B$ has a one-one degree strictly below
that of $A$ and is also bi-immune. Assume by contradiction that one
could one-one reduce $A$ to $B$ via some recursive $g$, then let $x_0 = 0$
and $x_{n+1} = g(x_n)+1$ for all $n$. For all $n$, $A(g(x_n)+1) = B(g(x_n))
= A(x_n)$ and $x_{n+1} \notin \{x_0,x_1,\ldots,x_n\}$, thus the
sequence $x_0,x_1,\ldots$ would become an infinite recursive enumeration
of numbers which are either all in $A$ or all outside $A$, in contradiction
to the bi-immunity of $A$. Furthermore, no recursive set is finite-one
reducible to any bi-immune set, as each recursive set is infinite
or coinfinite and thus, if the reduction would exist, $A$ would have
an infinite recursive subset or its complement would have an infinite
recursive subset.
\end{proof}

\noindent
The results obtained so far allow to classify the finite-one
degrees into three groups. The first and the third group have
uncountably many members each (given by cylinders and the biimmune
sets which both exist uncountably often and from each group, only
countably many can go into one finite-one degree) and the second
group consists only of two finite-one degrees, both contained in
the greatest recursive many-one degree.

\begin{corollary} \label{cor-finiteone}
Assume that $A$ is a set. Then for the finite-one degree of $A$, exactly one
of the following statements is true and each possibility can occur.
\begin{enumerate}
\item The finite-one and one-one degree of $A$ coincide and $A$ is in
  the greatest one-one degree of its many-one degree;
\item The finite-one degree of $A$ is an ascending chain isomorphic
  to the natural numbers with their natural ordering and consists either
  of all nonempty finite sets or of all cofinite sets with at least one
  nonelement;
\item Every recursive partial order (including an antichain) can be embedded
  into the finite-one degree of $A$ with each representative of the antichain
  being strictly one-one above $A$ itself; these finite-one degrees also do not
  contain a greatest one-one degree.
\end{enumerate}
\end{corollary}

\noindent
The collapse in the first case follows from the second item
in Theorem~\ref{thm-finiteone}. The second item is well-known
and the third item follows from Theorems~\ref{thm-one} and~\ref{thm-two}
and is explictely stated in Corrolary~\ref{cor-for-thm-two}.

Note that there are five recursive finite-one degrees, the two mentioned
in the second item of Corollary~\ref{cor-finiteone} and the
degrees $\{\emptyset\}$ and $\{{\mathbb N}\}$ and the finite-one
degree of all infinite and coinfinite recursive sets, the last three
consist all of a single one-one degree.

\section{Bounded finite-one degrees}

\noindent
Bounded finite-one reducibility is the variant of one-one
reducibility. Normally the word ``bounded'' means that the number
of queries are bounded by a constant,
following the notation used by Odifreddi \cite{Od1989,Od1999} and
others (like \cite{BHHV2002} for bounded reducibilities).
As Maslova \cite{Ma1979} called the finite-one reducibility
``bounded many-one reducibility'' or $bm$-reducibility,
the bounded finite-one reducibility would then be something
like ``constant bounded many-one reducibility'' in her terminology. 
Thus, as defined in Definition~\ref{def-reduction},
a bounded finite-one reduction $f$
from $A$ to $B$ is a finite-one reduction from $A$ to $B$ with the
additional constraint that there is a constant $c$ such that for
each $y$ there are at most $c$ numbers $x$ with $f(x) = y$. This
constant bound directly shows that bounded finite-one reducibilities
are $bm$ and $bm'$ reducibilities in the terminology of Maslova \cite{Ma1979}.
However, there is no direct way to make them to $bm''$-reducibilities.
Certain parallel properties are preserved when going from finite-one
to bounded finite-one degrees: There are exactly five such recursive
degrees. Furthermore, the join
$\{2x: x \in A\} \cup \{2x+1: x \in B\}$ of two sets $A,B$ represents
the least common upper bound in both, the finite-one degrees and the
bounded finite-one degrees and thus the degrees form an upper semilattice.
Furthermore, properties of recursively enumerable sets like being simple
or hypersimple are inherited downwards (except to recursive sets) along
finite-one and bounded finite-one reductions \cite{Ma1979}.
The next result shows, however, that
bounded finite-one degrees might consist of infinite ascending chains
of one-one degrees where each two of them are comparable; this differs
strongly from the case of finite-one degrees where the nonrecursive
degrees are either irreducible or contain antichains of one-one degrees.

\begin{rem} \label{rem:od-bfo}
The proof that a many-one degree either consists of a single one-one degree
or contains an infinite ascending chain of one-one degrees in Odifreddi's
book \cite[Proposition VI.6.13]{Od1989} works verbatim also for bounded
finite-one degrees; it is based on Young's corresponding result from 1966
\cite{Yo1966b}. Thus every bounded finite-one degree is either irreducible
and coincides with its one-one degree or contains an infinite ascending
chain of one-one degrees. The basic idea of the proof is that if $A$ is a
set which is not a cylinder then $A,A \oplus A, A \oplus A \oplus A, \ldots$
is an infinite strictly ascending chain in the one-one degrees; obviously
all these sets are bounded finite-one equivalent.
\end{rem}

\begin{theorem} \label{thm-boundedfiniteone}
There is a nonrecursive bounded finite-one degree such that the collection
of all its one-one degrees form an ascending chain of the same order type
as the natural numbers.
\end{theorem}

\begin{proof}
Let $E_0$ be a maximal set with complement $E_3$ and
split $E_0$ using the Sacks splitting theorem \cite{Sa1963a} into
two recursively enumerable sets $E_1$ and $E_2$ of incomparable
Turing degree; once that is done, one can choose using the
hyperimmune-free basis theorem of Jockusch and Soare \cite{JS1972}
a hyperimmune-free set $A$ such that $A$ is a superset of $E_1$ and its
complement a superset of $E_2$.

Now let $B$ be a set in the bounded finite-one
degree of $A$. One now shows that $B$ is one-one equivalent to the
join of $\ell$ copies of $A$ for some natural number $\ell \geq 1$.

To see this, one considers a bounded finite-one reduction $f$ from $B$ to $A$.
Furthermore, let $\ell$ be the biggest number such that for infinitely
many $y \in E_3$ there are $\ell$ different $x$ with $f(x) = y$; due to $f$
being bounded finite-one, such a maximal $\ell$ must exist. Furthermore,
$\ell > 0$ as otherwise $f$ maps almost all numbers either to $E_1$ or
to $E_2$ which allows to construct a decision-procedure for $B$, that is,
$B$ would be recursive and not in the bounded finite-one degree of $A$.
This implies that almost all members of $E_3$ are $\ell$ times in the
range of $f$ and moving finitely many elements from $E_3$ to $E_1$ and $E_2$
will result in all elements of the new
$E_3$ appearing exactly $\ell$ times in the range of $f$.
There are uniformly recursive ascending unions with
$E_1 = \cup_s E_{1,s}$, $E_2 = \cup_s E_{2,s}$ and $E_{1,0}$ and
$E_{2,0}$ being infinite.
For each $x$, let $y = f(x)$ and $k$ be the number such that $x$ is
the $k$-th number mapped to $y$ when looking at the $x$ in ascending
order. Furthermore, let $s$ be the least number such that one of the
following cases applies and define $h$ by the first case which applies.
\begin{enumerate}
\item If $y \in E_{1,s}$ then let $h(x)$ be the first
  element of $E_1 \times \{1,2,\ldots,k\}$ not yet in the range of $h$;
\item If $y \in E_{2,s}$ then let $h(x)$ be the first
  element of $E_2 \times \{1,2,\ldots,k\}$ not yet in the range of $h$;
\item If $k \leq \ell$ and $y \notin E_{1,s} \cup E_{2,s}$ 
  then let $h(x) = c(y,k)$.
\end{enumerate}
By construction, $h$ is one-one. Furthermore, the range of $h$ is
${\mathbb N} \times \{1,2,\ldots,\ell\}$,
as all numbers $y \in E_3$ are exactly $\ell$ times in the range
of $f$ and all numbers $y \in E_1 \cup E_2$ are eventually enumerated
into $E_1$ or $E_2$, respectively, and then mapped in a bijective way
to the target set following its enumeration. Thus $B$ is one-one
equivalent to the $\ell$-fold selfjoin of $A$. This directly implies
that the one-one degrees inside the bounded finite-one degree of $A$ are
linearly ordered, as $B$ is one-one reducible to $C$ whenever it
holds that $B$ is one-one equivalent to the $\ell$-fold selfjoin
of $A$ and $C$ is one-one equivalent to the $\ell'$-fold selfjoin
of $A$ and $\ell \leq \ell'$. It remains to show that the hierarchy
stands and that the bounded finite-one degree of $A$ is not irreducible.

Assume now by way of contradiction that $h'$ one-one reduces the
$\ell+1$-fold selfjoin $B$ of $A$ to the $\ell$-fold selfjoin $C$ of $A$.
Now define the following equivalence relation on the set of natural numbers:

Now let $\sim$ be the smallest equivalence relation which for all $x,y$
enforces $x \sim y$ whenever one of the below conditions holds:
\begin{enumerate}
\item $x = y$ (enforcing reflexiveness);
\item $y \sim x$ (enforcing symmetry); 
\item There are $k,z_1,z_2,\ldots,z_k$ with $x \sim z_1$ and $z_1 \sim z_2$
      and $\ldots$ and $z_k \sim y$ (enforcing transitiveness);
\item $x,y$ are both in $E_1$ (enforcing all of $E_1$ to go into
      one equivalence class);
\item $x,y$ are both in $E_2$ (enforcing all of $E_2$ to go into
      one equivalence class);
\item There are $i \in \{1,2,\ldots,\ell+1\}$ and $j \in \{1,2,\ldots,\ell\}$
      with $h'(c(x,i)) = c(y,j)$.
\end{enumerate}

\medskip
\noindent
So $\sim$ is an recursively enumerable equivalence class which respects
$A$, that is, $x \sim y$ implies that either both $x,y$ are in $A$
or none of $x,y$ is in $A$. The reason for this is that none of the
above rules enforces that an element of $A$ and a nonelement of $A$
become equivalent.

By the cohesiveness of $E_3$, an equivalence class of $\sim$ either
has a finite intersection with $E_3$ or contains almost all elements
of $E_3$. The latter cannot happen as both $A$ and its complement have
an infinite intersection with $E_3$. Thus all equivalence classes have
only a finite intersection with $E_3$. Therefore, only finitely many
elements of $E_3$ belong to the equivalence classes of $E_1$ and $E_2$
and there must be further finite equivalence classes which are a complete
subset of $E_3$. Let $E_4$ be such a finite equivalence class and let
$k$ be the number of its elements. Then $h'$ maps the $k \cdot (\ell+1)$
elements of $E_4 \times \{1,2,\ldots,\ell+1\}$ to the
$k \cdot \ell$ elements of $E_4 \times \{1,2,\ldots,\ell\}$.
Thus $h'$ cannot be one-one and therefore the bounded finite-one degree
of $A$ consists of infinitely many one-one degrees.
\end{proof}

\begin{rem}
If one does the above construction of Theorem~\ref{thm-boundedfiniteone}
with two Turing incomparable maximal sets $E_4$ and $E_5$ and then
considers $A \oplus B$ for the so obtained sets $A$ and $B$
of hyperimmune-free degree, then the one-one degrees of the
bounded finite-one degree of $A \oplus B$ consider of all $C_{i,j}$ being
joins of $i$ copies of $A$ and $j$ copies of $B$ with $i,j \geq 1$.
Now $C_{i,j}$ is one-one reducible to $C_{i',j'}$ if and only if
$i \leq i'$ and $j \leq j'$. Thus the one-one degrees are partially
ordered the same way as pairs of natural numbers and so contain finite
antichains of arbitrary length but no infinite antichains. An example
of a finite antichain of length $5$ is $(5,1),(4,2),(3,3),(2,4),(1,5)$
which translates into the antichain $C_{5,1},C_{4,2},\ldots,C_{1,5}$
of one-one degrees inside the bounded finite-one degree of $A \oplus B$.

Furthermore, the bounded finite-one degree of a Martin-L\"of random set $A$
contains an infinite antichain of one-one degrees where the representatives
are $B_x = A \oplus \{y: c(x,y) \in A\}$.
\end{rem}

\begin{rem} \label{rem-unbounded-vs-bounded}
It might be important to note that nonrecursive nonirreducible
finite-one degrees contain an antichain of bounded finite-one degrees. To
obtain this result, one modifies Theorem~\ref{thm-two} and its proof as
follows: First, in the numbering of all $\varphi_e$, one equips the functions
$\varphi_1,\varphi_2,\ldots$ with a boundedness-function $bound$ such that
for each $y$ there are at most $bound(e)$ many $x$ with $\varphi_e(x) = y$.
Note that introducing this bound makes the numbering nonacceptable, but it
still covers all reductions --- if for the first instance of the function,
$bound(e)$ was too small, then one produces another index $e'$ where then
$bound(e')$ is bigger. Functions intending to violate their bound are forced
to be partial in order to preserve the bound. In the light of this practice,
one could even just set $bound(e) = e$.

Furthermore, in Algorithm~\ref{algo-two} of the proof,
the corresponding sentence is adjusted to the following form:
\begin{quote}
While the number of $x$ with $h_i(x) = y$ is less or equal to the product
of $bound(e)$ and the number of $x'$ with $h_j(x') \leq y$ do begin
select the least $x$ where $h_i(x)$ is not yet defined and
let $h_i(x) = y$ end.
\end{quote}
With the adjustments following from this change in the proof, the proof
then actually shows that every recursive partial order can be embedded
effectively into every nonrecursive and nonirreducible finite-one degree
with respect to the bounded finite-one degrees inside that degree and not
just with respect to the one-one degrees inside it. The recursive
finite-one degrees coincide each with one bounded finite-one degree, thus
they form a special case, as two of these degrees have infinitely many
one-one degrees inside them.
\end{rem}

\section{Conclusion}

\noindent
The present work studies the collection of one-one degrees inside
finite-one degrees and the collection of finite-one degrees inside
many-one degrees. For finite-one degrees, it is shown that they
consist of a single one-one degree if they are the greatest finite-one
degree in their many-one degree; otherwise they consist of infinitely
many bounded finite-one degrees and contain an antichain of these
which implies that they also contain an antichain of one-one degrees.
This solves an open problem which was around since Young
\cite{Yo1966b} embedded the countable dense linearly ordered set into all
nonrecursive many-one degrees which consist of several one-one
degrees; Odifreddi \cite{Od1981,Od1989} stated this problem explicitly
as open. Significant progress towards this problem was done by
D\"egtev \cite{De1976} and Batyrshin \cite{Ba2021},
who proved the existence of antichains in all r.e.\ nonrecursive
nonirreducible many-one degrees and all limit-recursive nonrecursive
nonirreducible many-one degrees, respectively.

The present results answer this question in full generality and affirmatively.
Furthermore, the paper deepens the study of the finite-one degrees
and bounded finite-one degrees which are between one-one-degrees and
many-one degrees so that the following inclusion relation between these
degrees hold (for the degrees of a fixed set):
\begin{quote}
One-one degree $\subseteq$ bounded finite-one degree $\subseteq$
finite-one degree $\subseteq$ many-one degree.
\end{quote}
See \cite{Ba2021,BMO2019,HH1994,KW2021,Ma1979} for prior work on
finite-one reducibility in recursion theory and computational complexity.
The following paragraphs (a), (b) and (c) give an overview of the
results of the present work.

\medskip
\noindent
{\bf (a) Finite-one degrees inside many-one degrees.}
Every many-one degree consists of at least one and up to countably
many finite-one degrees; among those is a greatest finite-one degree
which coincides with its one-one degree (that is, it is an irreducible
one-one degree) and the many-one degree is irreducible if and only if
it has only one finite-one degree. The greatest recursive many-one degree
consists of three finite-one degrees which are those of all finite
nonempty sets, all cofinite sets with a nonempty complement and
all other recursive sets; the third degree is irreducible and
the other two degrees coincide
with bounded finite-one degrees which in turn are ascending chains of
one-one degrees and these ascending chains are
order-isomorphic to the natural numbers and their natural order $<$.
Theorem~\ref{thm-two} shows that nonrecursive and nonirreducible many-one
degrees have only one irreducible finte-one degree inside them, namely
the greatest one, and furthermore at least one nonirreducible further
finite-one degree, they are described under (b).
The number of finite-one degrees inside a nonirreducible many-one
degree is at least two and at most countable, where the greatest
recursive many-one degree gives the number three and other finite numbers
are open. Maslova \cite{Ma1979} showed that nonrecursive recursively enumerable
many-one degrees consist either of a single or of infinitely many
finite-one degrees.

Furthermore, Remark~\ref{rem-for-thm-two} points
out that the irreducible many-one degrees can be characterised as
those consisting of exactly one finite-one degree, thus providing one
possible characterisation for those requested by
Odifreddi \cite[Problem 4]{Od1981} who asked for criteria characterising
when many-one degrees are irreducible/nonirreducible.

\medskip
\noindent
{\bf (b) One-one degrees and Bounded finite-one degrees inside
finite-one degrees.}
Irreducible finite-one degrees coincide with their one-one degree.
Nonrecursive nonirreducible finite-one degrees satisfy that they
embed antichains and all other recursive partial orders by
Theorems~\ref{thm-one} and~\ref{thm-two}.
These results in particular show that
for nonrecursive and nonirreducible finite-one degrees, they allow
to embed by a uniformly recursive family
of invertible finite-one reductions a sequence of one-one degrees which
are ordered according to any given recursive partial order; one can furthermore
achieve that the so embedded one-one degrees are all not bounded finite-one
equivalent with each other, see Remark~\ref{rem-unbounded-vs-bounded}.
As every nonirreducible and nonrecursive many-one degree contains an
nonirreducible finite-one degree, this answers the question of
Odifreddi~\cite[Problem 4]{Od1981}.

\medskip
\noindent
{\bf (c) One-one degrees inside bounded finite-one degrees.}
The partial orders of one-one degrees inside finite-one degrees
(by one-one reducibility) can allow more variety than in the
case of finite-one or many-one degrees in the sense that more
different cases arise. Theorem~\ref{thm-boundedfiniteone} shows
that the one-one degrees inside a bounded finite-one degree can
be infinitely many which are linearly ordered --- the order type
is that of the natural numbers with $<$. Furthermore, there are
bounded finite-one degrees which have finite but no infinite antichains
and those which have infinite antichains.

\medskip
\noindent
{\bf Open Questions.}
The following questions on the structure of reducibilities between one-one and
many-one degrees are still open:
\begin{enumerate}
\item Does every nonrecursive many-one degree have a least finite-one
      degree inside?
\item Are there nonrecursive many-one degrees consisting of at least two
      but at most finitely many finite-one degrees? 
\item Are there bounded finite-one degrees consisting exactly of a dense
      linearly ordered set of one-one degrees?
\end{enumerate}

\medskip
\noindent
{\bf Further remarks.}
Recall that Theorem~\ref{thm-boundedfiniteone} provides a bounded finite-one
degree entirely consisting of linearly ordered one-one degrees forming an
ascending chain. So one might ask whether one also have other linear orders
than this specific one of the natural numbers? 

The construction of Theorem~\ref{thm-boundedfiniteone} was attempted to be
done more generally---Stephan \cite{St2001} studied important questions
about strong degrees and wanted in this paper also address other open
questions about strong degrees including the one addressed in
the current paper. His intention was to
construct a whole many-one degree in which the one-one degrees are linearly
ordered and together with Zhang---during his exchange to Singapore as an
undergraduate student---he tried another time to make this construction work.
Now Theorem~\ref{thm-one} shows that generalising
Theorem~\ref{thm-boundedfiniteone} to finite-one or many-one degrees
is impossible. So the open question is which linear orders can be realised
by the set of all one-one degrees inside a bounded finite-one degrees and
whether there are any besides the one-element linear order and the linear
order of natural numbers. The remarks after the theorem provide other partial
orders which are the orders of a bounded finite-one degree, but no further
linear order.

Concerning the number of finite-one degrees inside a many-one degrees,
Maslova \cite{Ma1979} was able to show that for recursively enumerable
degrees, this number is either $1$ or $\infty$, with the greatest recursive
degree being the only exception.

Note that finite-one degrees and bounded finite-one
degrees are less investigated than one-one
degrees and many-one degrees, thus more questions are open for
these than for the other two types of degrees; Batyrshin (in private
communication) pointed the authors to the work of Maslova. However,
also he is not aware of other systematic studies involving finite-one
or bounded finite-one degrees with respect to the world of sets.
A bit different is the picture for numberings. Furthermore, finite-one
functions are used frequently in topology and set theory to
define reductions in their fields.

There are two parallels between the study of numberings and very strong degrees:
Bazhenov, Mustafa and Ospichev \cite{BMO2019} considered also finite-one
reducibility (there called bounded reducibility due to the name used
by Maslova) to compare numberings of recursively enumerable sets.
The same could be done with numberings of recursive functions or
other objects of consideration. Bazhenov, Mustafa and Ospichev
showed that for finite-one reducibility, there is a uniformaly
recursively enumerable family with infinitely many members such that its
numberings form under this reducibility an infinite lattice; furthermore,
finite lattices with exactly $2^k-1$ members for $k=1,2,3,\ldots$ exist.
For the usual numberings under many-one reducibility, this is impossible,
as Khutoretskii \cite{Kh1971} showed that there are either only one or
infinitely many numberings (modulo many-one equivalence) and
Selivanov \cite{Se1976} showed that under many-one numberings, uniformly
recursively enumerable classes of sets form a lattice only if they just
consist of one numbering (modulo many-one equivalence).

The second parallel is that
the proof methods of Theorem~\ref{thm-one} and~\ref{thm-two}
construct arrays of one-one degrees inside the many-one degree of $A$
by constructing families of invertible finite-one reductions to $A$ ---
these reductions are constructed without any knowledge about $A$ beyond
these: $A$ is neither finite nor cofinite nor a cylinder. Only the
verification which makes the functions $f_d$ with cofinite domain total
in the case that the underlying $\varphi_e$ is a one-one reduction between
$B_i$ and $B_j$ which should not exist, uses finite amount of knowledge
about $A$ in order to make the $f_d$ a total strictly ascending self-reduction.
Such a self-reduction only exists when $A$ is a cylinder, as that was
a priori excluded, there are no one-one reductions $\varphi_e$ between
distinct $B_i,B_j$ in Theorem~\ref{thm-one}. This type of handling objects
enumerated without really knowing what they are has also been the practice
in various papers within the theory of numberings, for example,
the work of Goncharov, Lempp and Solomon~\cite{GLS2002}.

In some cases, there are further parallels between notions in recursion
theory and notions in computational complexity and automata theory. For
that reason, one-one reducibility and many-one reducibility are also
considered in computational complexity. Many-one reductions
between recursive sets which are infinite and co-infinite
can be replaced by finite-one reductions, as one can always reduce
to the largest element or non-element of the target set which is already
computed by the given reduction; so if targets $y$ and $z$ are both known
to be inside or outside the set, one can replace a reduction to $y$ by
one to $z$. Therefore, it is natural to consider finite-one reductions
in computational complexity; Hemaspaandra and Hemaspaandra \cite{HH1994}
provided such work and furthermore, Lischke \cite{Li1975,Li1976,Li1977}
also investigated complexity measures and reducibilities between them.

\medskip
\noindent
{\bf Acknowledgements.}
The authors are grateful to Ilnur Batyrshin and Guohua Wu for pointing out
these two parallels in the study of reducibilities between sets on one hand
and those between numberings. Furthermore, they thank both and Bj\o rn
Kjos-Hanssen for detailed discussions and pointing to the following references
of prior work on the topic of finite-one reducibility
\cite{BMO2019,BL1977,Ci2026,GLS2002,HH1994,KW2021,Li1975,Li1976,Li1977,Ma1979}.

\end{document}